\documentclass[12pt,letterpaper]{amsart}
\usepackage{geometry}
\usepackage{graphicx}
\usepackage{amsfonts}
\usepackage{amssymb}
\usepackage{mathrsfs,latexsym,url,graphicx}
\usepackage{mathpazo}
\usepackage{amsmath,amsthm,amscd}
\usepackage{setspace,cancel}
\newtheorem{theorem}{Theorem}
\newtheorem*{theorem*}{Theorem}
\newtheorem{proposition}{Proposition}
\newtheorem{lemma}{Lemma}
\newtheorem{corollary}{Corollary}

\theoremstyle{remark}
\newtheorem{remark}{Remark}
\newtheorem{definition}{Definition}
\newtheorem{example}{Example}

\usepackage{mathrsfs}
\usepackage{color}

\usepackage{amssymb,latexsym}

\newcommand{\VL}{\mathcal{L}}

\newcommand{\abs}[1]{\left\lvert#1\right\rvert}
\newcommand{\norm}[1]{\left\|#1\right\|}

\newcommand{\R}{\mathbb{R}}

\newcommand{\B}{\mathcal{B}}
\newcommand{\disc}{\mathbb{D}}
\newcommand{\C}{\mathbb{C}}

\newcommand{\Z}{\mathbb{Z}}
\newcommand{\cC}{\mathcal{C}}

\newcommand{\zb}{\overline{z}}
\newcommand{\D}{\Omega}
\newcommand{\Dc}{\overline{\Omega}}
\newcommand{\dbar}{\overline{\partial}}

\providecommand{\bysame}{\leavevmode\hbox to3em{\hrulefill}\thinspace}
\providecommand{\MR}{\relax\ifhmode\unskip\space\fi MR }
\providecommand{\href}[2]{#2}

\title[Obstructions for Compactness of Hankel Operators]{Obstructions for Compactness of Hankel Operators: Compactness Multipliers}
\author{Mehmet \c{C}el\.ik}
\address[Mehmet \c{C}elik]{Texas A\&M University - Commerce, Department of Mathematics, P.O. Box 3011
Commerce, Texas 75429}
\email{mehmet.celik@tamuc.edu}

\author{Yunus E. Zeytuncu}
\address[Yunus E. Zeytuncu]{University of Michigan - Dearborn, Department of Mathematics and Statistics, Dearborn, MI  48128}
\email{zeytuncu@umich.edu}

\date{\today}
\thanks{The work of the second author was partially supported by a grant from 
the Simons Foundation (\#353525). The work of the first
author was partially supported by Texas A\&M University-Commerce Provost Office and by L3-Communications, Greenville.}

\begin{document}
\maketitle

\begin{abstract}
We establish a connection between compactness of Hankel operators and geometry of the underlying domain through compactness multipliers for the $\dbar$-Neumann operator. In particular, we prove that any compactness multiplier induces a compact Hankel operator. We also generalize the notion of compactness multipliers to vector fields and matrices and then we use this generalization to generate compact Hankel operators.
\end{abstract}


\section{Introduction}


In this paper we study obstructions for compactness of Hankel operators on general pseudoconvex domains in $\C^n$. One of the most stimulating results for compactness of Hankel operators is due to Axler \cite{Axler1986}, which states that on the Bergman space of the unit disc the Hankel operator $H_{\overline{f}}$ for a holomorphic function $f$, is compact if and only if $f$ is in the little Bloch space $\B_0$. In particular, if the holomorphic function $f$ is additionally smooth up to the boundary then $f$ is automatically in the little Bloch space and hence the operator $H_{\overline{f}}$ is compact. However, on a general domain in $\C^n$ if we take a holomorphic symbol $f$ that is also smooth up to the boundary, we can not immediately conclude that the corresponding operator $H_{\overline{f}}$ is compact. In other words, in higher dimensions there is no universal characterization of compactness, and the geometry of the domain plays a decisive role. 

As for a more specific example, consider a convex domain in $\C^2$ that contains an analytic disc in its boundary. On these domains, if $g$ is smooth up to the boundary then the Hankel operator $H_g$ is compact if and only if $g$ is holomorphic along the analytic disc in the boundary, see \cite[Corollary 2]{CuckovicSahutoglu09}. This result indicates that one needs to investigate the boundary geometry of the underlying domain to understand compact Hankel operators.

The $L^2$ theory of the $\dbar$-Neumann operator is one of the common ways of relating the boundary geometry and compact Hankel operators, see \cite{Haslinger01}, \cite{CuckovicSahutoglu09}, \cite{Sahutoglu12}, \cite{CelikSahutoglu2012}, \cite{CuckovicSahutoglu13}, \cite{CuckovicSahutoglu14}, and \cite{ZeytuncuSahutoglu16} for some recent results. One reason for this connection is the Kohn's formula that conveniently links Hankel operators and the $\dbar$-Neumann operator. Another tool in this context is the notion of compactness multipliers, which have been studied with the purpose of characterizing  obstructions to compactness of the $\dbar$-Neumann problem \cite{CelikPhD08}. 
Therefore it is natural to relate these multipliers and obstructions for compactness of Hankel operators. As a first observation, on a bounded convex domain a function that is smooth up to the boundary of the domain is a compactness multiplier if and only if it vanishes on the closure of the union of all the (nontrivial) analytic discs in the boundary \cite{CelikStraube09}. On the same domain, a symbol function that is smooth up to the boundary induces a compact Hankel operator if and only if the symbol is holomorphic along analytic discs in the boundary \cite{CuckovicSahutoglu09}. Therefore, on convex domains any compactness multiplier induces a compact Hankel operator. Without the convexity assumption, such a connection between compactness multipliers and symbols of compact Hankel operators was less understood and the following specific question was posed in \cite{CuckovicSahutoglu09}.

\textit{Assume that $\D$ is a smooth bounded pseudoconvex domain in $\C^n$ and $f$ is a compactness multiplier, then is the Hankel operator $H_f$ compact on $\D$}?\\

In the first part of this paper, we establish a connection between compactness of Hankel operators and boundary geometry through compactness multipliers and as a consequence we answer the question above. 
\begin{theorem}\label{thm1}
Let $\D$ be a smooth bounded pseudoconvex domain in $\C^n$ and $f\in C(\Dc)$. If $f$ is a compactness multiplier then the Hankel operator $H_f$ is compact on $A^2(\D)$.
\end{theorem}

In the second part, we generalize the compactness multipliers tool to vector fields and matrices and then by using the generalized compactness multiplier device, we show how to generate symbols that induce compact Hankel operators.

\begin{theorem}\label{det-compute}
Let $\D$ be a smooth bounded pseudoconvex domain in $\mathbb{C}^n$ that admits a smooth plurisubharmonic defining
function $r$. Let
\begin{align}\label{Levi form}
\VL:=\left[\frac{\partial^{2}r}{\partial z_i \partial \zb_{j}}\right]_{1\leq i,j\leq n}
=\bordermatrix{\cr             
\cr           & \frac{\partial^{2} r}{\partial z_{1}\partial\zb_{1}}  & \frac{\partial^{2} r}{\partial z_{1}\partial\zb_{2}}  & ... &\frac{\partial^{2} r}{\partial z_{1}\partial\zb_{n}}  
\cr           & \vdots & \vdots  & &\vdots
\cr          & \vdots  & \vdots & &\vdots
\cr         & \frac{\partial^{2} r}{\partial z_{n-1}\partial\zb_{1}}  & \frac{\partial^{2} r}{\partial z_{n-1}\partial\zb_{2}}  &  ... &\frac{\partial^{2} r}{\partial z_{n-1}\partial\zb_{n}}
\cr          & \frac{\partial^{2} r}{\partial z_{n}\partial\zb_{1}}  & \frac{\partial^{2} r}{\partial z_{n}\partial\zb_{2}}  & ... &\frac{\partial^{2} r}{\partial z_{n}\partial\zb_{n}} \cr}
\end{align}
be the complex Hessian matrix. Then the Hankel operator $H_{\det(\VL)}$ is compact on $A^2(\D)$.
\end{theorem}

Finally, by following the ideas from \cite{Straube2008} we obtain a more general theorem that shows how to generate many more compact Hankel operators.

\begin{theorem}\label{general allowable}
Let $\D$ be as in Theorem \ref{det-compute}. If $A(z)$ is a positive semidefinite self-conjugate matrix (of entries continuous functions on $\Dc$) such that
for all $z\in \Dc$ and $\xi\in\C^{n}$
\begin{align}
0\leq \left(A^{m}(z)\cdot\xi\ ,\ \xi\right)\leq \left(\VL(z)\cdot\xi\ ,\ \xi\right)\ \ \text{for some}\ m\in \Z^{+}
\end{align}
then the Hankel operator $H_{\phi}$ is compact on $A^2(\D)$, where $\phi(z)=\det(A(z))$.
\end{theorem}

The remainder of the paper is organized as follows. In the second section, we recall the $L^2$ theory
of the $\dbar$-Neumann operator, we introduce compactness multipliers, allowable vector fields, allowable matrices, and the connection between the $\dbar$-Neumann operator and Hankel operators. In the same section, we relate the compactness multiplier notion to that of a symbol of a compact Hankel operator.
In Section $3$, we prove Theorem $1$, present a counterexample for its converse, and remark on the extensions to higher degree forms.
In Section $4$, we prove Theorems $2$ and $3$. Finally, in the last section we conclude with some examples and remarks.


\section{Definitions and Notations}
Let $\Omega\subset\Bbb{C}^n$ be a bounded, smooth pseudoconvex domain and let $\square=\dbar^{*}\overline{\partial}+\overline{\partial} \dbar^{*}$ be the $\overline{\partial}$-Neumann Laplacian, where $\overline{\partial}{}^{*}$ stands for the $L^{2}$-adjoint of the Cauchy-Riemann operator. Under these assumptions on $\Omega$, the operator $\overline{\partial}\colon L^2_{(0,q)}(\Omega)\rightarrow L^2_{(0,q+1)}(\Omega)$ is closed and densely defined. Furthermore, the operator $\square$ acting on its domain is invertible with a bounded inverse $N$, which is called the $\overline{\partial}$-Neumann operator, \cite{Hormander65}, see also \cite{ChenShawBook} and \cite{StraubeBook}. Compactness of the $\overline\partial$-Neumann problem is a basic property with many applications. When $b\D$ is smooth, compactness  implies global regularity of the $\dbar$-Neumann problem. Also, the Fredholm theory for Toeplitz operators is a direct consequence of the compactness of the $\dbar$-Neumann problem. The following estimate is a reformulation of the compactness property \cite[Chapter 4]{StraubeBook}.

A compactness estimate of the $\overline{\partial}$-Neumann operator is said to hold on $\Omega$ if for a given  $\varepsilon>0$ there is a constant $C_{\varepsilon}>0$ such that the following estimate 
\begin{eqnarray*}\label{eq1}
\norm{u}^{2}\leq \varepsilon\left(\norm{\overline{\partial}u}^{2}+\norm{\overline{\partial}^{*}u}^{2}\right)+C_{\varepsilon}\norm{u}_{-1}^{2}
\end{eqnarray*}
is valid $\forall u\in \mbox{Dom}(\overline{\partial})\cap\mbox{Dom}(\overline{\partial}^{*})\subset L_{(0,q)}^{2}(\Omega)$. ($\norm{\cdot}_{-1}$ is the $L^{2}$-Sobolev ($-1$)-norm.)

One way to investigate compactness of the $\dbar$-Neumann problem through the compactness estimate is to consider the set of functions $f$ for which one can estimate the $L^2$-norm of $fu$ in terms of $L^2$-norms of $\dbar u$, $\dbar^{*}u$, and the Sobolev $(-1)$-norm of $u$; the constant $C_{\varepsilon}$ is allowed to depend on $f$. Such functions are known as compactness multipliers \cite{CelikStraube09}. Compactness multipliers are inspired by the well-known subelliptic multipliers \cite{Kohn79}, with compactness estimates taking the place of the subelliptic estimates.

\begin{definition}
A function $f\in \cC(\overline{\Omega})$ is called a \textit{compactness multiplier} on $\Omega$
if for every  $\varepsilon>0$ there is a constant $C_{\varepsilon,f}>0$ such that the following estimate
\begin{eqnarray}\label{eq:no1a}
\norm{fu}^{2} \leq \varepsilon ( \norm{\overline{\partial}u}^{2}+\norm{\overline{\partial}^{*}u}^{2})+ C_{\varepsilon,f}\norm{u}_{-1}^{2}
\end{eqnarray}
is valid $\forall u\in \mbox{Dom}(\overline{\partial})\cap \mbox{Dom}(\overline{\partial}^{*})\subset L_{(0,q)}^{2}(\Omega)$.
\end{definition}
\begin{remark}\label{real c.m.}
$f$ is a compactness multiplier if and only if $\overline{f}$ is a compactness multiplier. Then, the real and imaginary parts of $f$ are compactness multipliers. As a result, it is sufficient to consider real valued compactness multipliers. 
\end{remark}

Compactness multipliers have been studied with the purpose of characterizing the obstructions to compactness of the $\dbar$-Neumann problem. Let $J^q$ be the set of compactness multipliers associated with $(0,q)$-forms, $1\le q\le n$, and denote by $A_q$ the common zero set of the elements in $J^q$, i.e. $J^q=\{f\in C(\overline{\Omega})|\ f\equiv 0 \text{ on } A_q\}$. Then the $\dbar$-Neumann operator is compact if and only if $A_q$ is empty \cite{CelikStraube09}. In the same paper, it was also showed that on bounded convex domains, the set $A_q$ is exactly the closure of the union of $q$-dimensional analytic disks in the boundary.\\

In our work, we will also use the derivatives of the compactness multipliers. For a real valued function, we set $\sum_{j=1}^{n}\frac{\partial f}{\partial z_j}u_j$ to be the interior product (the adjoint of exterior multiplication) of the vector field $\overline{\partial f}=\sum_{j=1}^n\overline{\left(\frac{\partial f}{\partial z_j}\right)}\frac{\partial}{\partial \zb_j}$ and the $(0,1)$-form $u=\sum_{j=1}^{n}u_j d\zb_j$. One can estimate the $L^2$-norm of $\sum_{j=1}^{n}\frac{\partial f}{\partial z_j}u_j$ in terms of $L^2$-norms of $\dbar u$, $\dbar^{*} u$, and the Sobolev $(-1)$-norm of $u$; again constant $C_{\varepsilon}$ is allowed to depend on $\partial f$.

We will also employ the term ``\emph{allowable}'' which was used in the study of subelliptic multipliers by D'Angelo in \cite{D'AngeloBook1992}.
\begin{definition}
A vector field
\[
v=\sum_{j=1}^{n}v_{j}\frac{\partial}{\partial z_{j}}
\]
of type $(1,0)$  
is \textit{allowable} (for compactness estimate of the 
$\dbar$-Neumann problem) if for every $\varepsilon>0$ there is $C_{\varepsilon}>0$ such that 
\begin{align}\label{AllowableVector}
\left\|\sum_{j=1}^{n}v_{j}u_{j}\right\|^{2} &\leq \varepsilon ( \norm{\overline{\partial}u}^{2}+\norm{\overline{\partial}^{*}u}^{2})+ C_{\varepsilon}\norm{u}_{-1}^{2}
\end{align}
for all $u=\sum_{j=1}^{n}u_{j}d\zb_{j}\in \mbox{Dom}(\overline{\partial})\cap \mbox{Dom}(\overline{\partial}^{*})\subset L_{(0,1)}^{2}(\Omega)$.\\ 
We will also need the following notation. 
\[M_{v}:\mbox{Dom}(\overline{\partial})\cap \mbox{Dom}(\overline{\partial}^{*})\rightarrow L_{(0,0)}^{2}(\Omega)\] 
\[M_{v}(u):=\sum_{j=1}^{n} v_j u_{j},\ \ \text{ where }\ \ u=\sum_{j=1}^{n}u_j d\zb_j.\]
Thus, the estimate \eqref{AllowableVector} can be rewritten as
\begin{align}
\left\|M_{v}(u)\right\|^{2}&\leq \varepsilon ( \norm{\overline{\partial}u}^{2}+\norm{\overline{\partial}^{*}u}^{2})+ C_{\varepsilon}\norm{u}_{-1}^{2}
\end{align}
for all $u=\sum_{j=1}^{n}u_{j}d\zb_{j}\in \mbox{Dom}(\overline{\partial})\cap \mbox{Dom}(\overline{\partial}^{*})\subset L_{(0,1)}^{2}(\Omega)$. 
\end{definition}
In the next definition we generalize the definition of an \emph{allowable vector field} to an \emph{allowable matrix}.

\begin{definition}\label{allowable matrix}
An $n\times n$ matrix $A(z)$  with smooth entries on $\C^n$ is called an \textit{allowable matrix} if for each $\varepsilon >0$ there is a constant $C_{\varepsilon}>0$ such that.
\begin{align}\label{matrix6}
\norm{A(z)u}^{2}\leq \varepsilon ( \norm{\overline{\partial}u}^{2}+\norm{\overline{\partial}^{*}u}^{2})+ C_{\varepsilon}\norm{u}_{-1}^{2}
\end{align}
for all $u=\sum_{k=1}^{n}u_{k}d\zb_{k}\in \mbox{Dom}(\overline{\partial})\cap \mbox{Dom}(\overline{\partial}^{*})\subset L_{(0,1)}^{2}(\Omega)$
\end{definition}

$\norm{A(z)u}^2$ in \eqref{matrix6} represents $\sum_{j=1}^n\left\|\sum_{k=1}^n A_{jk}(z)u_k\right\|^{2}$.
Note that this definition is equivalent to saying that each row of the matrix is an \emph{allowable row}. The definition allows the replacement of a row's entries with components of an allowable vector field.\\

Let $A^2(\Omega)$ be the subspace of holomorphic functions in $L^2(\Omega)$. The operator $P\colon L^2(\Omega) \longrightarrow A^2(\Omega)$ denotes the Bergman projection. The Hankel operator with symbol $\psi \in L^\infty(\Omega)$ is the operator defined as $$ H_\psi = (I - P)\psi \colon A^2(\Omega) \longrightarrow L^2(\Omega), $$ where the symbol $\psi$ is identified with the corresponding multiplication operator. 
   Using Kohn's formula $P = I - \overline{\partial}{}^{*} N_1 \overline{\partial}$, the following relation between the $\overline{\partial}$-Neumann operator $N$ and the Hankel operator $H_{\psi}$ with symbol $\psi \in C^1(\Omega)$ is obtained $ H_\psi (f) = \overline{\partial}{}^{*} N_1 \overline{\partial} (\psi f) = \overline{\partial}{}^{*} N_1 (f \overline{\partial}\psi ),\quad \forall f \in A^2(\Omega). $
   Because of this formula it is natural to expect strong connections between $N_1$ and $H_{\psi}$.

\section{Proof of Theorem 1}
Let $H_{\phi}^{\D}$ denote the Hankel operator on $\D$ with symbol $\phi$ and $R_U$ be the restriction operator onto an open set $U$. One can still use functions from $A^2(\D)$ and work locally, on a neighborhood $\D\cap U$ of $p\in b\D$ by using the composition of the Hankel and restriction operators, $H_{R_{\D\cap U}(\phi)}^{\D\cap U}R_{\D\cap U}$. The composition of those two operators is well-defined on $A^2(\D)$.

On bounded pseudoconvex domains the set of compactness multipliers in $C(\Dc)$ is a closed ideal $J$ and we can express $J=\{f\in C(\Dc)\ |\ f|_{K}\equiv 0\}$, where $K$ denotes the common zero set of the elements in $J$. Moreover, $K$ is a subset of the set of infinite type points on $b\D$, $(b\D)_{\infty}$, and it can easily be shown that the set $(b\D)_{\infty}\backslash K$ is benign for the compactness of the $\dbar$-Neumann operator \cite{CelikPhD08}.

\begingroup
\def\thetheorem{\ref{thm1}}
\begin{theorem}
Let $\D$ be a smooth bounded pseudoconvex domain in $\C^n$ and $f\in C(\Dc)$. If $f$ is a compactness multiplier then the Hankel operator $H_f$ is compact on $A^2(\D)$.
\end{theorem}
\addtocounter{theorem}{-1}
\endgroup
\begin{proof}
We start with the localization of Hankel operators technique introduced in \cite[Proposition $1$, $(ii)$]{CuckovicSahutoglu09}. First, we show that for every $p\in b\D$ there is an open neighborhood $U_p$ such that $\D\cap U_p$ is a domain, and $H_{R_{\D\cap U_p}(f)}^{\D\cap U_p}R_{\D\cap U_p}$ is compact on $A^2(\D)$, then by using \cite[Proposition $1$, $(ii)$]{CuckovicSahutoglu09}, we conclude that $H_{f}^{\D}$ is compact on $A^2(\D)$.
For this purpose, we look at the points $p\in b\D$ in two separate cases.\\

\emph{Case 1:} If $p\in b\D \backslash K$, then there is a complex ball $B(p,r)$ centered at $p$ with radius $r>0$ such that $B(p,r)\cap K=\emptyset$. Define 
$$U_p:=B(p,r)\cap\D.$$ 
The $\dbar$-Neumann operator is compact on $L^2(U_p)$. The set of infinite type points, not in $K$, is benign for the compactness of the $\dbar$-Newmann operator \cite{CelikPhD08} and the rest of the boundary points are of finite type. 

If the $\dbar$-Neumann operator is compact on $L^2(U_p)$ then $R_{U_p}(f)\in C(\overline{U_p})$ (for all $f\in C(\Dc)$) is a compactness multiplier and $H_{R_{U_p}(f)}^{U_p}R_{U_p}$ is a compact operator on $A^2(\D)$ for all $f\in C(\Dc)$.\\

\emph{Case 2:} As for the points in $K$, we use the following construction.
For each $j\geq 1$ define an open set \[U_j:=\{z\in\C^n\ |\ \text{dist}(z,K)<1/j\}\]
containing the zero set $K$ and choose $\chi_j(z)\in C^{\infty}(\C^n)$ such that $0\leq\chi_j(z)\leq 1$, $\chi_j(z)\equiv 1$ on $\C^n\backslash U_j$, and $\chi_j(z)=0$ on $U_{2j}$. Now, define $f_j(z):=\chi_j(z)\cdot f(z)$, where $f(z)\in C(\Dc)$ is a compactness multiplier.

Thus, for every $j\geq 1$, $\{f_j(z)\}\subset C(\Dc)$ such that $f_j(z)\equiv 0$ on $U_{2j}$ and $f_j(z)$ is a compactness multiplier (because $K\subset U_{2j}$). Moreover, for every $j\geq 1$, $H_{R_{U_{2j}\cap\D}(f_j)}^{U_{2j}\cap\D}R_{U_{2j}\cap\D}\equiv 0$ and is a compact operator on $A^2(\D)$.\\
\\
Thus, by \emph{Cases} $1$ and $2$ we conclude that for every $p\in b\D$ there is an open neighborhood $U_p$ such that $\D\cap U_p$ is a domain, and $H_{R_{\D\cap U_p}(f_j)}^{\D\cap U_p}R_{\D\cap U_p}$ is compact on $A^2(\D)$. 
By \cite[Proposition $1$, $(ii)$]{CuckovicSahutoglu09} we conclude that $H_{f_j}^{\D}$ is compact on $A^2(\D)$.\\

Now, the idea is to approximate $f$ uniformly on $\Dc$ by the above constructed sequence of functions $\{f_j\}$.
To show that $H_f^{\D}$ is a compact operator on $A^2(\D)$, it is enough to see that the Hankel operators $\left\{H_{f_j}^{\D}\right\}$ converge to $H_f^{\D}$ in operator norm. Indeed, 
\begin{align}
H_f^{\D}-H_{f_j}^{\D}=(I-P)(f)-(I-P)(f_j)=(f-f_j)I-P(f-f_j)
\end{align}
and the operator norm of the multiplication by $(f-f_j)$ on $L^2(\D)$ is $\text{max}_{z\in\Dc}|(f-f_j)|$. Thus,

\begin{align}
\|H_f^{\D}-H_{f_j}^{\D}\|\rightarrow 0,\ \text{ as }\ j\rightarrow\infty. 
\end{align}

\end{proof}

\begin{remark}
If $H_{f}$ is a compact Hankel operator on $A^2(\D)$ then its symbol function $f$ is not necessarily a compactness multiplier. For example, $H_{z_1}\equiv 0$ and so is compact on $A^2(\D)$ where $\D$ is a smoothed bi-disc, however $z_1$ is not a compactness multiplier because $z_1\not=0$ on $\{(z_1,z_2)\ :\ 0\leq|z_2|\leq 1/2\ \text{ and }\ |z_1|=1 \}$. See Example \ref{ex1} for more details.
\end{remark}


\begin{remark}
A Hankel operator with a symbol function $f$ is equal to the negative of the commutator operator with the multiplication symbol $f$ and Bergman projection $P$, $H_f(u)=(I-P)(fu)=-[P,f](u)$. The result in Theorem \ref{thm1} also applies for the commutator operator $[P,f]$ on $A^2(\D)$. We conclude that on a smooth bounded pseudoconvex domain $\D$, if $f\in C(\Dc)$ is a compactness multiplier then the commutator operator $[P,f]$ is compact on $A^2(\D)$. Moreover, by employing \cite[Corollary 2]{CelikSahutoglu2014} we further deduce that the commutator operator $[P_q,f]$ is compact on $A_{(0,q)}^2(\D)$ for all $0\leq q\leq n-1$, where $P_q$ denotes the orthogonal projection from $L_{(0,q)}^2(\D)$ onto the subspace of $\dbar$-closed forms.
\end{remark}

\section{Compactness Multipliers Machinery}
A subelliptic estimate of the $\dbar$-Neumann problem is a stronger condition than a compactness estimate. Thus, every subelliptic multiplier is also a compactness multiplier but the converse is false. Indeed, consider a convex domain $\D$ in $\C^2$ and on the boundary of this domain have a set of infinite type points with empty Euclidean interior, any smooth function on $\Dc$ not vanishing on the boundary of the domain, $b\D$, is a compactness multiplier, but not a subelliptic multiplier. Kohn \cite{Kohn79} developed subelliptic multipliers and created an algorithmic procedure for computing certain ideals. He used these ideals to find out if there is a complex analytic variety in the boundary and if there is a subelliptic estimate. Creating an analogue algorithmic procedure with ideals of compactness multipliers for compactness estimate can be helpful to examine obstructions for the compactness property of the $\dbar$-Neumann operator and of Hankel operators. However, it is important to note that Kohn's algorithm is in space of real analytic functions. The ring of real analytic functions is a Noetherian ring, where every prime ideal of the ring is finitely generated. This property of the ring of functions plays a fundamental role in Kohn's algorithm, it determines when such ideals define trivial varieties on the boundary of the domain. 
Since the $\dbar$-Neumann operator and its compactness property very much depend on the boundary geometry of the domain, the role of the defining function in the theory of compactness multipliers becomes fundamental. The defining function itself being a compactness multiplier helps us to connect the geometry of the domain with the compactness multiplier notion. Because we are working on domains with smooth boundaries (not necessarily real analytic) this forces us to work with compactness multipliers from the ring of smooth functions. However, the ring of smooth functions is not Noetherian. Absence of this essential property makes establishing an analogous algorithm with compactness multipliers challenging. In this write up, we ignore questions relating to the algorithmic point of view.
Instead, we study the connections of compactness multipliers with a symbol function of a compact Hankel operator.

The following proposition establishes a connection between a compactness multiplier and an allowable vector field, analogous to subelliptic multiplier case, see  \cite{Kohn79} or \cite{D'AngeloBook1992}. 
\begin{proposition}\label{p2}
Let $\D$ be a smooth bounded pseudoconvex domain in $\C^n$. Suppose that $f\in C^{2}(\overline{\D})$ is a compactness multiplier. Then $\partial f$ is an allowable vector field. That is, for every $\varepsilon>0$ there exists $C_{\varepsilon,\partial f}>0$ such that  
\[\left\|\sum_{j=1}^{n}\frac{\partial f}{\partial z_{j}}u_{j}\right\|^{2}\leq \varepsilon\left(\norm{\overline{\partial}u}^{2}+\norm{\overline{\partial}^{*}u}^{2}\right)+ C_{\varepsilon,\partial f}\norm{u}_{-1}^{2}\]
for all $u\in\text{Dom}(\dbar^*)\cap\text{Dom}(\dbar)\subset L_{(0,1)}^2(\D)$. 
\end{proposition}

\begin{remark}
Consequently, for $1\leq q\leq n$, the operator 
$M_{\partial f}:\mbox{Dom}(\overline{\partial})\cap \mbox{Dom}(\overline{\partial}^{*})\left(\subset L_{(0,q)}^2(\D)\right)\rightarrow L_{(0,q-1)}^{2}(\Omega)$
is a compact operator, providing a property weaker than compactness of the $\dbar$-Neumann operator.
\end{remark}

\begin{corollary}\label{def-allowable}
Let $\D$ be a smooth bounded pseudoconvex domain in $\C^n$. Let $r$ be a smooth defining function for $\D$. Then the vector field
$\partial r=\sum_{j=1}^{n}\frac{\partial r}{\partial z_{j}}\frac{\partial}{\partial z_{j}}$ is allowable.
\end{corollary}

\begin{proof}[ Proof of Corollary \ref{def-allowable}]
To see that $r$ is a compactness multiplier check with \cite[Remark 2]{CelikStraube09} (or \cite[Proposition 4]{CelikPhD08}) and then by Proposition \ref{p2} we obtain that $\partial r$ is an allowable vector field.
 
\end{proof}


\begin{proof}[Proof of Proposition \ref{p2}]
Initially, we work with smooth $(0,1)$-forms in Dom$(\dbar^*)$ and then at the end of the proof we use the Density Lemma \cite[Lemma 4.3.2]{ChenShawBook} (or \cite[Proposition 2.3]{StraubeBook}) to move the result to Dom$(\dbar)\cap$Dom$(\dbar^*)$. Thus, let $u\in C_{(0,1)}^{\infty}(\D)\cap \text{Dom}(\dbar^*)$. \\

Let $\psi:=\sum_{j=1}^{n}\frac{\partial f}{\partial z_{j}}u_{j}$, so
\begin{align}\label{eq1}
\left\|\sum_{j=1}^{n}\frac{\partial f}{\partial z_{j}}u_{j}\right\|^{2}&=\left(\sum_{j=1}^{n}\left(\frac{\partial (fu_{j})}{\partial z_{j}}-f\frac{\partial u_{j}}{\partial z_{j}}\right),\psi\right)\\
\nonumber &=\left(\sum_{j=1}^{n}\frac{\partial (fu_{j})}{\partial z_{j}},\psi\right)+\left(-f\sum_{j=1}^{n}\frac{\partial u_{j}}{\partial z_{j}},\psi\right)
\end{align}
Now, let's estimate the last term,
\begin{align}\label{eq10}
\nonumber \abs{\left(-\sum_{j=1}^{n}\frac{\partial u_{j}}{\partial z_{j}},\overline{f}\psi\right)}=\abs{\left(\dbar^{*}u,\overline{f}\psi\right)}&\leq \left\|\ \dbar^{*}u\right\|\ \norm{f\psi}\\
 &\leq (a/2)\norm{\overline{\partial}^{*}u}^{2}+1/(2a)\norm{f\psi}^{2}\ \text{ for any }\ a>0. 
\end{align}
The second inequality follows from the small constant-large constant inequality\footnote{We refer to $(a/2)$ as a small constant, $1/(2a)$ as a large constant.}. As for the last term $\norm{f\psi}^{2}$, first notice that 
\[\psi=\sum_{k=1}^n \frac{\partial f}{\partial z_k} u_k=\sum_{k=1}^n\frac{\partial}{\partial z_k}\left(fu_k\right)-\sum_{k=1}^n f\frac{\partial u_k}{\partial z_k}=\sum_{k=1}^n\left[\frac{\partial}{\partial z_k},f\right]u_k,\] 
and
\begin{align*}
f\psi=f\sum_{k=1}^n\left[\frac{\partial}{\partial z_k},f\right]u_k&=f\sum_{k=1}^n \frac{\partial f}{\partial z_k} u_k\\
&=2f\sum_{k=1}^n \frac{\partial f}{\partial z_k} u_k+f^2\sum_{k=1}^n \frac{\partial u_k}{\partial z_k}-f\sum_{k=1}^n \frac{\partial f}{\partial z_k} u_k-f^2\sum_{k=1}^n \frac{\partial u_k}{\partial z_k}\\
&=\sum_{k=1}^n\frac{\partial}{\partial z_k}\left(f^2 u_k\right)-f\sum_{k=1}^n\frac{\partial}{\partial z_k}\left(fu_k\right)\\
&=\sum_{k=1}^n\left[\frac{\partial}{\partial z_k},f\right](fu_k)
\end{align*}
As an operator, $\sum_{k=1}^n\left[\frac{\partial}{\partial z_k},f\right]$ is a zeroth-order pseudo-differential operator, so 
\[\left\|f\psi\right\|^{2}=\left\|\sum_{k=1}^n\left[\frac{\partial}{\partial z_k},f\right]\left(fu_k\right)\right\|^{2}\lesssim \left\|\sum_{k=1}^n fu_k\right\|^{2}.\] 
Second, since $f$ is a compactness multiplier (from the hypothesis) for every $\varepsilon>0$ there exists $C_{\varepsilon, f}>0$ such that
\begin{align}\label{eq11}  
\left\|f\psi\right\|^{2}\lesssim\left\|fu\right\|^{2}\leq\varepsilon\left(\norm{\overline{\partial}u}^{2}+\norm{\overline{\partial}^{*}u}^{2}\right)+ C_{\varepsilon,f}\norm{u}_{-1}^{2}
\end{align}
for all $u\in\text{Dom}(\dbar^*)\cap\text{Dom}(\dbar)\subset L_{(0,1)}^2(\D)$. When we put this estimate in \eqref{eq10} we obtain
\begin{align}\label{eq12}
\abs{\left(-\sum_{j=1}^{n}\frac{\partial u_{j}}{\partial z_{j}},\overline{f}\psi\right)}\leq (a/2)\norm{\overline{\partial}^{*}u}^{2}+1/(2a)\left(\varepsilon\left(\norm{\overline{\partial}u}^{2}+\norm{\overline{\partial}^{*}u}^{2}\right)+ C_{\varepsilon,f}\norm{u}_{-1}^{2}\right) \text{ for any }\ a>0. 
\end{align}

To handle the remaining term $\sum_{j=1}^{n}\left(\frac{\partial (fu_{j})}{\partial z_{j}},\psi\right)$ in \eqref{eq1} we use  integration by parts,
\begin{align}\label{eq13}
\sum_{j=1}^{n}\left(\frac{\partial (fu_{j})}{\partial z_{j}},\psi\right)&=\sum_{j=1}^{n}\left(-fu_{j},\frac{\partial \psi}{\partial \zb_{j}}\right)+\sum_{j=1}^{n}\int_{b\Omega}f\frac{\partial r}{\partial z_{j}}u_{j}\overline{\psi}\\
\nonumber &=\sum_{j=1}^{n}\left(-fu_{j},\frac{\partial \psi}{\partial \zb_{j}}\right).
\end{align}
The boundary integral in the integration by parts vanishes because $u$ is in the domain of $\overline{\partial}^{*}$.
Then,
\begin{align}\label{eq14}
\sum_{j=1}^{n}\left(\frac{\partial (fu_{j})}{\partial z_{j}},\psi\right) &=\sum_{j=1}^{n}\left(-fu_{j},\frac{\partial \psi}{\partial \overline{z}_{j}}\right)\\
\nonumber &=\sum_{j=1}^{n}\left(-fu_{j},\frac{\partial}{\partial \overline{z}_{j}}\left(\sum_{k=1}^{n}\frac{\partial f}{\partial z_{k}}u_{k}\right)\right)
\end{align}
Now, by the Cauchy-Schwartz inequality we get
\begin{align}
\nonumber  \sum_{j=1}^{n}\left(-fu_{j},\frac{\partial}{\partial \overline{z}_{j}}\left(\sum_{k=1}^{n}\frac{\partial f}{\partial z_{k}}u_{k}\right)\right)&\leq \sum_{j=1}^{n}\left\|fu_{j}\right\|\ \sum_{j=1}^{n}\left\|\frac{\partial}{\partial\overline{z}_{j}}\left( \sum_{k=1}^{n}\frac{\partial f}{\partial z_{k}}u_{k}\right)\right\|
\end{align}
Then by the small constant-large constant inequality we obtain
\begin{align}
\nonumber  \sum_{j=1}^{n}\left\|fu_{j}\right\|\ \sum_{j=1}^{n}\left\|\frac{\partial}{\partial\overline{z}_{j}}\left( \sum_{k=1}^{n}\frac{\partial f}{\partial z_{k}}u_{k}\right)\right\|&\leq 1/(2a)\sum_{j=1}^{n}\norm{fu_{j}}^{2}+(a/2)\sum_{j=1}^{n}\left\|\frac{\partial}{\partial\overline{z}_{j}}\left( \sum_{k=1}^{n}\frac{\partial f}{\partial z_{k}}u_{k}\right)\right\|^{2}
\end{align}
An estimate for the second term follows from the bar derivatives of $u$, which are controlled (in $L^2(\D)$) by $\dbar u$ and $\dbar^* u$. That is,
\begin{align}\label{eq15}
 1/(2a)\sum_{j=1}^{n}\norm{fu_{j}}^{2}+(a/2)\sum_{j=1}^{n}\left\|\frac{\partial}{\partial\overline{z}_{j}}\left( \sum_{k=1}^{n}\frac{\partial f}{\partial z_{k}}u_{k}\right)\right\|^{2}&\lesssim 1/(2a)\norm{fu}^{2}+(a/2)\left(\norm{\overline{\partial}u}^{2}+\norm{\overline{\partial}^*u}^{2}\right).
\end{align}
When we put this estimate back in \eqref{eq14} and \eqref{eq13} we get
\begin{align*}
\left| \sum_{j=1}^{n}\left(\frac{\partial (fu_{j})}{\partial z_{j}},\psi\right)\right| \lesssim 1/(2a)\norm{fu}^{2}+(a/2)\left(\norm{\overline{\partial}u}^{2}+\norm{\overline{\partial}^*u}^{2}\right).
\end{align*}
When we use \eqref{eq11} for the first term we get
\begin{align}\label{eq16}
\left| \sum_{j=1}^{n}\left(\frac{\partial (fu_{j})}{\partial z_{j}},\psi\right)\right| \lesssim 1/(2a)\left(\varepsilon\left(\norm{\overline{\partial}u}^{2}+\norm{\overline{\partial}^{*}u}^{2}\right)+ C_{\varepsilon,f}\norm{u}_{-1}^{2}\right)+(a/2)\left(\norm{\overline{\partial}u}^{2}+\norm{\overline{\partial}^*u}^{2}\right).
\end{align}

Finally, we put \eqref{eq16} and \eqref{eq12} into \eqref{eq1} to get 
\begin{align}\label{eq17}
\left\|\sum_{j=1}^{n}\frac{\partial f}{\partial z_{j}}u_{j}\right\|^{2}&\lesssim (a/2)\left(\norm{\overline{\partial}u}^{2}+\norm{\overline{\partial}^*u}^{2}\right)\\
\nonumber &+(a/2)\norm{\overline{\partial}^{*}u}^{2}+1/a\left(\varepsilon\left(\norm{\overline{\partial}u}^{2}+\norm{\overline{\partial}^{*}u}^{2}\right)+ C_{\varepsilon,f}\norm{u}_{-1}^{2}\right)\\
\nonumber &\leq \left(a+\frac{\varepsilon}{a}\right)\left(\norm{\overline{\partial}u}^{2}+\norm{\overline{\partial}^{*}u}^{2}\right)+\frac{C_{\varepsilon,f}}{a}\norm{u}_{-1}^{2}
\end{align}
We set $a=\sqrt{\varepsilon}$ and choose $\varepsilon$ small enough to get the desired estimate in Proposition \ref{p2}.

\end{proof}


Recall that an \emph{allowable matrix} is the same thing as saying that each row of the matrix is an \emph{allowable row}. For example, if each $j^{\text{th}}$-row of a matrix is constructed from components $\{v_{ij}\}$ of an allowable vector field $v_j=\sum_{i=1}^{n} v_{ij}\frac{\partial}{\partial z_i}$ then the matrix is an allowable matrix.

\begin{proposition}\label{determinant}
Let $\D$ be a smooth bounded pseudoconvex domain in $\C^n$. If $A(z)$ is an allowable matrix with $A_{ij}(z)\in
\cC^{\infty}(\overline{\Omega})$ then the determinant of the matrix $A$ is a compactness multiplier.
\end{proposition}
\begin{proof}

Let $A(z)$ be such an allowable matrix. Then $B(z):=A^{*}(z)A(z)$ is a Hermitian positive semi-definite matrix, $B_{ij}(z)\in \cC^{\infty}(\overline{\Omega})$ and there exists strictly positive $C(z)\in\cC^{\infty}(\overline{\Omega})$ such that
\begin{align}\label{nn1}
(\det\left(B(z)\right)u,u)_{E}\leq C(z)(B(z)u,u)_{E},
\end{align}
where the inner product $(\cdot,\cdot)_E$ is the standard one in $\C^n$ and do not involve integration. 

Indeed, if $\det\left(B(z)\right)=0$ then \eqref{nn1} holds trivially. If $\det\left(B(z)\right)>0$, that is $B(z)$ is a Hermitian positive definite matrix, then there exists strictly positive $C'(z)\in\cC^{\infty}(\overline{\Omega})$ such that 
$$C'(z)\abs{\xi}^{2}\leq \left(B(z)\xi,\xi\right)_E$$ 
for all $\xi\in \C^{n}$. Multiply both sides by $\frac{\det\left(B(z)\right)}{C'(z)}$ to have 
$$\det\left(B(z)\right)\abs{\xi}^{2}\leq \frac{\det\left(B(z)\right)}{C'(z)}\left(B(z)\xi,\xi\right)_{E}.$$ 
Set $C(z)=\frac{\det\left(B(z)\right)}{C'(z)}$ to get \eqref{nn1}. We also set $C=\max_{\overline\D}C(z)$, 
then considering the inequality \eqref{nn1} and $u\in \cC_{(0,1)}^{\infty}(\D)\cap \text{dom}(\overline{\partial}^{*})$, we have
\begin{align*}
(\det(A(z)A^{*}(z))u,u)_{E}\leq C(A^{*}(z)A(z)u,u)_{E}=C\abs{A(z)u}_{E}^2.
\end{align*}
Then
\begin{align*}
\norm{\det\left(A(z)\right)u}_{L^{2}}^{2}=\sum_{j=1}^{n}\int\limits_{\D} \abs{\det(A(z))}_{E}^{2}\abs{u_{j}}^{2}\leq C\int\abs{A(z)u}_{E}^{2}
=C\norm{A(z)u}_{L^{2}}^{2}.
\end{align*}
Since $A(z)$ is an allowable matrix, we have $\det\left(A(z)\right)$ as a compactness multiplier.

\end{proof}

\begin{remark}
A special set of allowable matrices for the compactness estimate of the $\dbar$-Neumann problem has already been studied by Straube in \cite{Straube2008}.
\end{remark}

\subsection{Complex Hessian and Compact Hankel Operators}

We also present ways of producing other symbols which induce compact Hankel operators on the same domain. First, we present an installment of a fundamental symbol, developed completely from the defining function, more specifically, complex Hessian, carrying information about the boundary geometry of the domain. Second, we present how to identify more compact Hankel operators by using derivatives of the symbols in an iterative sense, with every iteration of the derivative one will be able to generate another compact Hankel operator on the domain.

\begingroup
\def\thetheorem{\ref{det-compute}}
\begin{theorem}
Let $\D$ a smooth bounded pseudoconvex domain in $\C^n$ that admits a smooth plurisubharmonic defining function $r$ such that $|dr|=1$ on $b\Omega$. Then $H_{\det(\VL)}$ 
is a compact Hankel operator on $A^2(\D)$, where $\VL$ is the complex Hessian matrix.
\end{theorem}
\addtocounter{theorem}{-1}
\endgroup

\begin{proof}
If the determinant of $\VL$ is a compactness multiplier, then by Theorem \ref{thm1} we conclude that $H_{\det(\VL)}$ is a compact Hankel operator on $A^2(\D)$. Since $\Omega$ is a pseudoconvex domain, $\VL$ is positive semidefinite matrix. Then, by the same linear algebra idea in Proposition \ref{determinant} we can find a constant $C>0$ such that
\begin{align*}
\left(\left|\det(\VL)\right|^{2}u,u\right)_{E}\leq C \cdot(\VL u,u)_{E}
\end{align*}
and so 
\[\norm{\det(\VL)u}^{2}=\int_{\Omega}( \det(\VL)u,\det(\VL)u)=\int_{\Omega}(\left|\det(\VL)\right|^{2}u,u)\leq C \int_{\Omega}(\VL u,u)_E.\] 

Thus, to show the determinant of $\VL$ is a compactness multiplier we need to show that for a given $\varepsilon>0$ there is a $C_{\varepsilon}>0$ such that

\[\int_{\Omega}\left(\VL u,u\right)_E\lesssim \varepsilon \left(\norm{\overline{\partial}u}^{2}+\norm{\overline{\partial}^{*}u}^{2}\right)+C_{\varepsilon}\norm{u}_{-1}^{2}\]

$\forall u\in C_{(0,1)}^{\infty}\left(\D\right)\cap \text{dom}\left(\overline{\partial}^{*}\right)$. 

This estimate has already been proven by Straube in \cite{Straube2008}. We go over the proof for convenience.

Note that  $\int_{\Omega}(\VL u,u)=\int_{\Omega}\sum_{j,k=1}^{n}\frac{\partial^{2} r(z)}{\partial z_{j}\partial \overline{z}_{k}} u_{j}\overline{u}_{k}\ dV(z)$. \footnote{$A\lesssim B$ if $\exists c>0$ such that $A\leq cB$.}
Let's split $u$ into its tangential and normal parts near $b\D$, $u=u_N+u_T$. Let $\eta(z)$ be a cutoff function whose support is contained in a $\mu$-neighborhood of $b\D$, $\{z\in\Dc\ |\  -\mu <r(z)<\mu\}$; $\eta(z)\equiv 1$ on $\{z\in\Dc\ |\  -\mu/2 \leq r(z)\leq\mu/2\}$; and $u=(1-\eta)u+\eta u_N+\eta u_T$.  

\begin{align}\label{eq:no5:4c}
\nonumber\int_{\Omega}\sum_{j,k=1}^{n}&\frac{\partial^{2}r(z)}{\partial z_{j}\partial \overline{z}_{k}}\ u_{j}(z)\overline{u}_{k}(z)\ dV(z)\\
\nonumber&=\int_{\Omega}\sum_{j,k=1}^{n}\frac{\partial^{2}r(z)}{\partial z_{j}\partial \overline{z}_{k}}\ \left[(1-\eta)u+\eta u_N+\eta u_T\right]_{j}(z)\cdot\overline{\left[(1-\eta)u+\eta u_N+\eta u_T\right]}_{k}(z)\ dV(z)\\
&\lesssim \int_{\Omega}\sum_{j,k=1}^{n}\frac{\partial^{2}r(z)}{\partial z_{j}\partial \overline{z}_{k}}\ (\eta u_{T})_{j}(z)\ \overline{(\eta u_{T})}_{k}(z)\ dV(z)
+ \norm{\eta u_{T}}\ \norm{\eta u_{N}}+\norm{\eta u_{N}}^{2}+\norm{(1-\eta)u}^2.
\end{align}
and then by using the large constant-small constant inequality on the second term on the right we get
\begin{align*}
\norm{\eta u_{T}}\ \norm{\eta u_{N}}\lesssim (a/2)\norm{\eta u_{T}}^{2} + 1/(2a)\norm{\eta u_{N}}^{2}.
\end{align*}
Thus, the right hand side of \eqref{eq:no5:4c} can be bounded from above by (we replace $(a/2)$ with $\varepsilon$  and $1/(2a)$ with $C_{\varepsilon}$ )
\begin{align}\label{eq:no5:4d}
\lesssim \int_{\Omega}\sum_{j,k=1}^{n}\frac{\partial^{2}r(z)}{\partial z_{j}\partial \overline{z}_{k}}\ (\eta u_{T})_{j}(z)\ \overline{(\eta u_{T})}_{k}(z)\ dV(z)
+\varepsilon\norm{\eta u_{T}}^2+ C_{\varepsilon}\norm{\eta u_{N}}^2+\norm{(1-\eta)u}^2.
\end{align}
$\norm{(1-\eta)u}^2$ is supported at the interior points of $\Omega$, so it has a compactness estimate by the interior elliptic regularity. For the $C_{\varepsilon}\norm{\eta u_{N}}^2$ term we use interpolation inequality between Sobolev norms (that is, the estimate $\norm{u}^2\leq\varepsilon\norm{u}_1^2+C_{\varepsilon}\norm{u}_{-1}^2$), and that the normal component is having Sobolev $1$- subelliptic estimate (that is, $\norm{u_{N}}_{1}\lesssim \norm{\overline{\partial}u}+\norm{\overline{\partial}^{*}u}$. Moreover, $\norm{\eta u_{T}}^2\leq\norm{u}^2\lesssim \norm{\overline{\partial}u}^2+\norm{\overline{\partial}^{*}u}^2$. Thus, the right hand side of (\ref{eq:no5:4d}) is bounded by 

\begin{align}\label{eq:no5:4e}
 \lesssim\int_{\Omega}\sum_{j,k=1}^{n}\frac{\partial^{2}r(z)}{\partial z_{j}\partial \overline{z}_{k}}\ (\eta u_{T})_{j}(z)\ \overline{(\eta u_{T})}_{k}(z)\ dV(z)+ \varepsilon (\norm{\overline{\partial}u}^{2}+\norm{\overline{\partial}^{*}u}^{2})+C_{\varepsilon}\norm{u}_{-1}^{2}.
\end{align}
As for the first term of (\ref{eq:no5:4e}): Let $\D_{\delta}:=\{z\in\D\ |\  d(z,b\D)<-\delta\}$ for $0\leq\delta\leq \varepsilon$. Note that $u_{T}$ is in $\text{dom}(\dbar^*)$ on $\Omega_{\delta}$ by the definition of $u_T$. Thus, we split the integral, 
\[\int_{\Omega}\sum_{j,k=1}^{n}\frac{\partial^{2}r(z)}{\partial z_{j}\partial \overline{z}_{k}}\ (\eta u_{T})_{j}(z)\ \overline{(\eta u_{T})}_{k}(z)\ dV(z)\] into two integrals
\begin{align}\label{eq:no5:4f}
\int_{\Omega\diagup\Omega_{\varepsilon}}\sum_{j,k=1}^{n}\frac{\partial^{2}r(z)}{\partial z_{j}\partial \overline{z}_{k}}\ (\eta u_{T})_{j}(z)\ \overline{(\eta u_{T})}_{k}(z)\ dV(z)
+\int_{\Omega_{\varepsilon}}\sum_{j,k=1}^{n}\frac{\partial^{2}r(z)}{\partial z_{j}\partial \overline{z}_{k}}\ (\eta u_{T})_{j}(z)\ \overline{(\eta u_{T})}_{k}(z)\ dV(z).
\end{align}
The second term on the right is supported at the interior points of $\Omega$ so it has an  estimate by the interior elliptic regularity.
The first term on the right can be estimated by the help of Kohn-Morrey formula and the Fubini's theorem.
\begin{align}\label{eq:no5:4g}
\int_{\Omega\diagup\Omega_{\varepsilon}}\sum_{j,k=1}^{n}\frac{\partial^{2}r(z)}{\partial z_{j}\partial \overline{z}_{k}}&\ (\eta u_{T})_{j}(z)\ \overline{(\eta u_{T})}_{k}(z)\ dV(z)\\
\nonumber&=
\int_{0}^{\varepsilon}\int_{b\Omega_{\delta}}\sum_{j,k=1}^{n}\frac{\partial^{2}r(z)}{\partial z_{j}\partial \overline{z}_{k}}\ (\eta u_{T})_{j}(z)\ \overline{(\eta u_{T})}_{k}(z)\ d\sigma(z)\ d\delta \\
\nonumber&\leq \int_{0}^{\varepsilon}(\norm{\overline{\partial}(\eta u_{T})}^{2}+\norm{\overline{\partial}^{*}(\eta u_{T})}^{2})\ d\delta
\end{align}
Then, the right side of (\ref{eq:no5:4g}) is bounded from above by $\varepsilon\ \left(\norm{\overline{\partial} u}^{2}+\norm{\overline{\partial}^{*} u}^{2}\right)$.

Therefore, from \eqref{eq:no5:4e}, \eqref{eq:no5:4f}, and \eqref{eq:no5:4g} we have 
\begin{align}\label{eq:no23a}
\int_{\Omega}&\sum_{j,k=1}^{n}\frac{\partial^{2}r(z)}{\partial z_{j}\partial \overline{z}_{k}}\ u_{j}(z)\ \overline{u}_{k}(z)\ dV(z)\\
\nonumber &\lesssim \varepsilon\ (\norm{\overline{\partial}\ u)}^{2}+\norm{\overline{\partial}^{*}u}^{2})+C_{\varepsilon}\norm{u}_{-1}^{2}
\end{align}
\end{proof}

\begin{corollary}\label{det mult}
Let $\D$ be a smooth bounded pseudoconvex domain in $\C^n$ that admits a smooth plurisubharmonic defining function $r$.
If $f\in C^{2}(\Dc)$ is a compactness multiplier then $H_{\det(B)}$ is a compact Hankel operator on $A^2(\D)$, where $B$ is a matrix created by replacing the $j^{\text{th}}$- row of the complex Hessian matrix with the components of the vector field $\partial f=\sum_{j=1}^n\frac{\partial f}{\partial z_j}\frac{\partial}{\partial z_j}$:
\[B:=\bordermatrix{\cr 
\cr           & \frac{\partial^{2} r}{\partial z_{1}\partial\zb_{1}}  & \frac{\partial^{2} r}{\partial z_{1}\partial\zb_{2}}  & ... &\frac{\partial^{2} r}{\partial z_{1}\partial\zb_{n}}  
\cr          & \vdots  & \vdots & \vdots&\vdots
\cr         & \frac{\partial^{2} r}{\partial z_{j-1}\partial\zb_{1}}  & \frac{\partial^{2} r}{\partial z_{j-1}\partial\zb_{2}}  &  ... &\frac{\partial^{2} r}{\partial z_{j-1}\partial\zb_{n}}
\cr &\frac{\partial f}{\partial z_1} & \frac{\partial f}{\partial z_2} & ... & \frac{\partial f}{\partial z_n}
\cr         & \frac{\partial^{2} r}{\partial z_{j+1}\partial\zb_{1}}  & \frac{\partial^{2} r}{\partial z_{j+1}\partial\zb_{2}}  &  ... &\frac{\partial^{2} r}{\partial z_{j+1}\partial\zb_{n}}
\cr          & \vdots  & \vdots & \vdots&\vdots
\cr          & \frac{\partial^{2} r}{\partial z_{n}\partial\zb_{1}}  & \frac{\partial^{2} r}{\partial z_{n}\partial\zb_{2}}  & ... &\frac{\partial^{2} r}{\partial z_{n}\partial\zb_{n}} \cr}.
\] 
\end{corollary}
\begin{proof}[Proof of Corollary \ref{det mult}]
$f\in C^{2}(\Dc)$ is a compactness multiplier then by Proposition \ref{p2} we have $\partial f$ is an allowable vector field (for a compactness estimate). We will show in Corollary \ref{allowable Levi} that the complex Hessian matrix is allowable, that is, every row is an allowable vector field. \\

The matrix $B$ is  formed by replacing the entire $j^{\text{ th}}$ row of the complex Hessian matrix (in Theorem \ref{det-compute}) with components of $\partial f$, $\left(\frac{\partial f}{\partial z_1},\frac{\partial f}{\partial z_2},...,\frac{\partial f}{\partial z_n},\right)$. The matrix $B$ is  allowable. Then, by Proposition \ref{determinant} the determinant of the matrix $B$ is a compactness multiplier. Finally, by using Theorem \ref{thm1} we conclude that $H_{\det(B)}$ is a compact Hankel operator.

\end{proof}

Let $f_1,f_2,...,f_n\in C^{2}(\Dc)$ and 
\[F:=\bordermatrix{\cr 
\cr           & \frac{\partial f_1}{\partial z_{1}}  & \frac{\partial f_1}{\partial z_{2}}  & ... &\frac{\partial f_1}{\partial z_{n}}  
\cr
\cr           & \frac{\partial f_2}{\partial z_{1}}  & \frac{\partial f_2}{\partial z_{2}}  & ... &\frac{\partial f_2}{\partial z_{n}}
\cr
\cr          & \vdots  & \vdots & \vdots&\vdots
\cr
\cr           & \frac{\partial f_n}{\partial z_{1}}  & \frac{\partial f_n}{\partial z_{2}}  & ... &\frac{\partial f_n}{\partial z_{n}} \cr}.
\] 

\begin{corollary}\label{det mult 2}
Let $\D$ be a smooth bounded pseudoconvex domain in $\C^n$.
If $f_1,f_2,...,f_n\in C^{2}(\Dc)$ are compactness multipliers then 
$H_{\det(F)}$ is a compact Hankel operator on $A^2(\D)$.
\end{corollary}
\begin{proof}
The proof is the same as in Corollary \ref{det mult}.

\end{proof}

\begin{remark}
In Corollary \ref{det mult 2} we can replace one of the compactness multipliers in the hypothesis with the resulting compactness multiplier $\det(F)$ and then apply Corollary \ref{det mult 2} to get another compact Hankel operator.
\end{remark}

Corollary \ref{det mult} presents a connection between two compact Hankel operators with two different symbols: one of the symbols is just making the operator compact and the other symbol involves the characteristics of the domain itself (the second symbol partially involves the complex Hessian matrix of the domain). Furthermore, Corollary \ref{det mult 2} presents a connection between a set of compact Hankel operators with another compact Hankel operator whose symbol is completely developed from the gradients of the symbols of the other compact Hankel operators.

\subsection{Allowable Matrices and their relation to Hankel operators}

In this subsection we further investigate allowable matrices for the compactness estimate. After developing some technical work, we formulate Corollary \ref{allowable Levi} and Theorem \ref{general allowable}. 

\begin{proposition}\label{sym}
If $A\in \R^{n\times n}$ is a symmetric positive semi-definite matrix then there exist a unique matrix $X\in\R^{n\times n}$ satisfying the equation $X^{2}=A$. Moreover, $X$ is also symmetric positive semi-definite.
\end{proposition}
\begin{proof}
See \cite[proof of Theorem 7.2.6.]{MatrixAnalysis}.

\end{proof}

\begin{lemma}\label{lemma1}
Let $B(z)$ and $D(z)$ be two symmetric positive semi-definite matrices with entries continuous on $\Dc$. 
If for all $z\in \Dc$ and $\xi\in\C^{n}$
\begin{align*}
0\leq \left(B(z)\cdot\xi\ ,\ \xi\right)\leq \left(D(z)\cdot\xi\ ,\ \xi\right)
\end{align*}
and for all $\varepsilon>0$ there is $C_{\varepsilon}>0$ such that 
\begin{align}\label{fake comp est}
\left\langle D(z)u,u \right\rangle_{L^{2}}\leq \varepsilon ( \norm{\overline{\partial}u}^{2}+\norm{\overline{\partial}^{*}u}^{2})+ C_{\varepsilon}\norm{u}_{-1}^{2}.
\end{align}
then $B(z)$ is an allowable matrix.
\end{lemma}
\begin{proof}
The square root of the matrix $B(z)$ is unique by Proposition \ref{sym} and has a compactness estimate on $\D$:

\[\norm{B^{1/2}u}^{2}=\int_{\Omega}(Bu)\overline{u}=\left\langle B u,u\right\rangle_{L^2}\leq \left\langle D(z)u,u \right\rangle_{L^{2}}.\] 
Then by the hypothesis of the theorem we have the estimate for all $\varepsilon>0$ there is $C_{\varepsilon}>0$ such that 
\begin{align}\label{sqrt est1}
\norm{B^{1/2}u}_{L^{2}}^{2}\leq \varepsilon ( \norm{\overline{\partial}u}^{2}+\norm{\overline{\partial}^{*}u}^{2})+ C_{\varepsilon}\norm{u}_{-1}^{2}.
\end{align}

It follows that,
\begin{align*}
\norm{B u}_{L^{2}}^{2}&=\left(B u,B u\right)_{L^{2}}\\
\nonumber &=\left(B^{1/2}B^{1/2}u,B u\right)_{L^{2}}\\
\nonumber &=\left(B^{1/2}u,\overline{(B^{1/2})}^{t}B u\right)_{L^{2}}\\
\nonumber &\leq \norm{B^{1/2}u}_{L^{2}}\ \norm{\overline{(B^{1/2})}^{t}B u}_{L^{2}}\\
\nonumber &\leq 1/(2a)\norm{B^{1/2}u}_{L^{2}}^{2} + (a/2)\norm{\overline{(B^{1/2})}^{t}B u}_{L^{2}}^{2}\\
\nonumber &\leq \varepsilon ( \norm{\overline{\partial}u}^{2}+\norm{\overline{\partial}^{*}u}^{2})+ C_{\varepsilon}\norm{u}_{-1}^{2}.
\end{align*}
In the first inequality, we use the Cauchy-Schwarz inequality. In the second one, we use  the large constant-small constant inequality. 
As for the last inequality, the second term on the right is under control because the entries of the matrix $B(z)$ are smooth on $\overline{\D}$ and the norm of the matrix $B(z)$ is bounded on $\overline{\D}$. Then the small constant represented as $(a/2)$ will take care of the constant coming from $\overline{(B^{1/2}(z))}^{t}B(z)$ in front of the norm and the term will be estimated with $(a/2) ( \norm{\overline{\partial}u}^{2}+\norm{\overline{\partial}^{*}u}^{2})$. As for the first term $1/(2a)\norm{B^{1/2}u}_{L^{2}}^{2}$, we use the above estimate \eqref{sqrt est1}.

\end{proof}

As an application of Lemma \ref{lemma1} one can see that on a smooth bounded pseudoconvex domain $\Omega$ in $\C^{n}$ the complex Hessian matrix $\VL$ (see eq. \eqref{Levi form}) is an allowable matrix.

\begin{corollary}\label{allowable Levi}
Let $\Omega$ be a smooth bounded pseudoconvex domain in $\C^{n}$ that admits a smooth plurisubharmonic defining function. Then the complex Hessian matrix $\VL$ is an allowable matrix. 
\end{corollary}

\begin{proof}
Consider $B(z)=D(z)=\VL$ (in Lemma \ref{lemma1}) to get the result. Note that the estimate \eqref{fake comp est} in the hypothesis of Lemma \ref{lemma1} for $\VL$ is showed in the proof of Theorem \ref{det-compute}.
\end{proof}

\begin{remark}
Combining Corollary \ref{allowable Levi} and Proposition \ref{determinant} it follows that $\det(\VL)$ is a compactness multiplier and thus the result in Theorem \ref{det-compute} follows. Note that in Theorem \ref{det-compute} we do not use that complex Hessian is an allowable matrix, but we use an inequality relation between a matrix and its determinant. 
\end{remark}

A more general application of Lemma \ref{lemma1} is the  following. 

\begingroup
\def\thetheorem{\ref{general allowable}}
\begin{theorem}
Let $\D$ be as in Theorem \ref{det-compute}. If $A(z)$ is a positive semidefinite self-conjugate matrix (of entries continuous functions on $\Dc$) such that
for all $z\in \Dc$ and $\xi\in\C^{n}$
\begin{align}
0\leq \left(A^{m}(z)\cdot\xi\ ,\ \xi\right)\leq \left(\VL(z)\cdot\xi\ ,\ \xi\right)\ \ \text{for some}\ m\in \Z^{+}
\end{align}
then the Hankel operator $H_{\phi}$ is compact on $A^2(\D)$, where $\phi(z)=\det(A(z))$. 
\end{theorem}
\addtocounter{theorem}{-1}
\endgroup

\begin{proof} If $A(z)$ is allowable, by Proposition \ref{determinant} we see that $\det(A(z))$ is a compactness multiplier and then by Theorem \ref{thm1} it follows that $H_{\det\left(A(z)\right)}$ is a compact operator on $A^2(\D)$. 
Therefore, it suffices to show that for any cases of $m\in \Z^{+}$ the matrix $A(z)$ is allowable.\\

If the hypothesis of the theorem is satisfied for $m=1$,  we have 
\begin{align}\label{ineq1}
0\leq \left\langle A(z) u,u\right\rangle_{L^{2}}&\leq \left\langle \VL(z) u, u\right\rangle_{L^{2}}
\end{align}
$\forall u\in \cC_{(0,1)}^{\infty}(\D)\cap \text{dom}(\dbar^*)$. It follows from Lemma \ref{lemma1} that the matrix $A(z)$ is allowable. Note that the estimate \eqref{fake comp est} in the hypothesis of Lemma \ref{lemma1} for $\VL$ is showed in the proof of Theorem \ref{det-compute}.\\

With the hypothesis $0\leq (A^{m}(z) u,u)_{L^{2}}\leq (\VL(z) u, u)_{L^{2}}$ for some $m\in \Z^{+}$ and $\forall u\in \cC_{(0,1)}^{\infty}(\D)\cap \text{dom}(\dbar^*)$, we show that $A^m$ is an allowable matrix by using the same approach as for $m=1$:\\

To complete the proof it is enough to show that for given $m\in \Z^{+}$ and $\varepsilon>0$, there exists $C_{\varepsilon}>0$ such that 
\begin{align}\label{estReviewer}
\norm{Au}_{L^{2}}^{2}&\leq \varepsilon\norm{A^{m}u}_{L^{2}}^{2}+C_{\varepsilon,m}\norm{u}_{L^{2}}^{2}\\
\nonumber &\forall u\in \cC_{(0,1)}^{\infty}(\D)\cap \text{dom}(\dbar^*)
\end{align}
If $m=2$ then this simply follows from the small constant-large constant inequality. Indeed, 
\begin{align*}
\norm{Au}_{L^{2}}^{2}&=\left(Au,Au \right)=\left(A^{*}Au,u \right)_{L^{2}}\\
\nonumber {}_{{}_{(A^{*}=A)}}&=\left(A^{2}u,u \right)_{L^{2}}\\
\nonumber &\leq \norm{A^{2}u}_{L^{2}}\ \norm{u}_{L^{2}}\\
\nonumber &\leq \frac{1}{2a}\norm{A^{2}u}_{L^{2}}^{2}+\frac{a}{2}\norm{u}_{L^{2}}^{2}
\end{align*}
For case $m=3$ we are going to use the following interpolation inequality; $$\norm{A^{2}u}^{2}=(A^{2}u,A^{2}u)=(A^{3}u,Au)\leq\norm{A^{3}u}\ \norm{Au}$$
which, actually can be generalized as follows; 
\begin{align}\label{gen interpolation}
\norm{A^{\frac{r+l}{2}}u}^{2}\leq\norm{A^{r}u}\ \norm{A^{l}u}
\end{align}
so we have 
\begin{align}
\norm{Au}^{2}&\leq\frac{1}{2a}\norm{A^{2}u}^{2}+ \frac{a}{2}\norm{u}^{2}\\
\nonumber &\leq \frac{1}{2a}\norm{A^{3}u}\norm{Au}+\frac{a}{2}\norm{u}^{2} ~\text{ (one more lc-sc inequality)}\\
\nonumber &\leq \frac{1}{2a}\left(\frac{1}{4a}\right)\norm{A^{3}u}^2+\frac{1}{2a}a\norm{Au}^{2}+\frac{a}{2}\norm{u}^{2}.
\end{align}
We subtract the middle term from both sides, and multiply both sides by 2 to get
\[  \norm{Au}^{2}\leq \frac{1}{4a^2}\norm{A^{3}u}^{2}+a\norm{u}^{2}.\]
When we choose $a$ sufficiently large we we get \eqref{estReviewer}.  The general case $m>3$ is virtually the same where the generalized interpolation inequality \eqref{gen interpolation} is used.

\end{proof}


\section{Examples and Remarks}

We start with the classical example of a rounded polydisc.
\begin{example}\label{ex1}
Let 
\[\lambda(t)=0\ \text{if}\ t\leq 0\ \&\ \lambda(t)=e^{-1/t}\ \text{if}\ t>0\]
$\lambda$ is a convex function on $(-\infty,1/2)$.
Then, consider the following domain
\[\D:=\left\{(z_1,z_2)\ |\ \rho(z_1,z_2)=\lambda\left(\frac{1}{2}\left(|z_1|^2-\frac{1}{4}\right)\right)+\lambda\left(\frac{1}{2}\left(|z_2|^2-\frac{1}{4}\right)\right)-e^{-8/3}<0\right\}.\]
\begin{center} 
\includegraphics[width=0.35\textwidth]{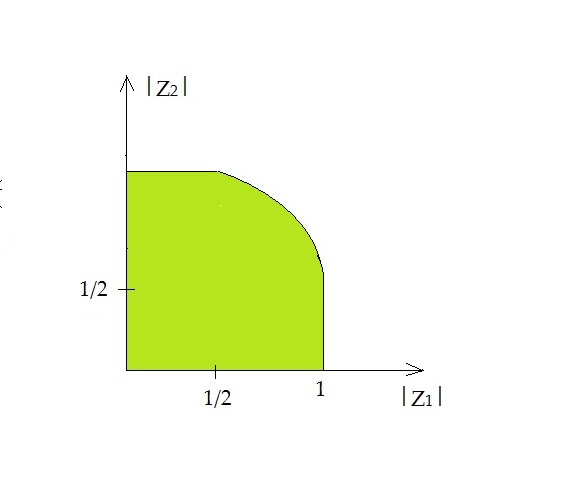} 
\end{center}
$\D$ is a smooth bounded pseudoconvex domain. Let
\[L_1:=\{(z_1,z_2)\ :\ 0\leq|z_2|\leq 1/2\ \text{ and }\ |z_1|=1 \}\] and
\[L_2:=\{(z_1,z_2)\ :\ 0\leq|z_1|\leq 1/2\ \text{ and }\ |z_2|=1 \}\]
then \[L:=L_1\cup L_2\subset b\D\]
is (Levi flat) foliated with analytic discs.\\
Note that the defining function of $\D$, $\rho(z_1,z_2)$ is a compactness multiplier: $\rho(z_1,z_2)$ vanishes on $L$.  Theorem \ref{thm1} implies that $H_{\rho}$ is compact on $A^2(\D)$. 

\begin{align*}
\VL:&=\left[\frac{\partial^{2}\rho}{\partial z_i \partial \zb_{j}}\right]_{1\leq i,j\leq 2}\\
&=\bordermatrix{\cr             
\cr           & \frac{(-2|z_1|^4+4|z_1|^2+1/8)}{\left(|z_1|^2-1/4\right)^4}\lambda\left(\frac{1}{2}(|z_1|^2-1/4)\right) &  0
\cr & &
\cr          & 0& \frac{(-2|z_2|^4+4|z_2|^2+1/8)}{\left(|z_2|^2-1/4\right)^4}\lambda\left(\frac{1}{2}(|z_2|^2-1/4)\right)  \cr}
\end{align*}
is the complex Hessian matrix and by Theorem \ref{det-compute} the matrix $\VL$ is allowable. The determinant of $\VL$ is
\begin{align*}
\det(\VL)=k(z_1,z_2)\lambda\left(\frac{1}{2}(|z_1|^2-1/4)\right)\lambda\left(\frac{1}{2}(|z_2|^2-1/4)\right).
\end{align*}
where $k(z_1,z_2)=\frac{4(|z_1|^2-1)^2(|z_2|^2-1)^2-17/8(|z_1|^2-1)^2-17/8(|z_2|^2-1)^2+17^2/18^2}{\left(|z_1|^2-1/4\right)^4 \left(|z_2|^2-1/4\right)^4}$.\\
\\
$\det(\VL)$ vanishes on the set $L$ and so Proposition \ref{determinant} gives us that $\det(\VL)$ is a compactness multiplier and by Theorem \ref{thm1} it follows that the Hankel operator $H_{\det(\VL)}$ is compact on $A^2(\D)$.

Let's form matrix $B$ by replacing the second row of the complex Hessian matrix $\VL$ with the allowable vector field $\partial\rho$.
\begin{align*}
B:&=
\bordermatrix{\cr
\cr           & 
\frac{\partial^2 \rho}{\partial z_1\partial \zb_1} & \frac{\partial^2 \rho}{\partial z_1\partial \zb_2}
\cr & &
\cr & \frac{\partial \rho}{\partial z_1} & \frac{\partial \rho}{\partial z_2}
\cr}\\
&=
\bordermatrix{\cr
\cr           & \frac{(-2|z_1|^4+4|z_1|^2+1/8)}{\left(|z_1|^2-1/4\right)^4}\lambda\left(\frac{1}{2}(|z_1|^2-1/4)\right)  &  0
\cr & &
\cr &\frac{2\zb_1}{(|z_1|^2-1/4)^2}\lambda\left(\frac{1}{2}\left(|z_1|^2-\frac{1}{4}\right)\right)& \frac{2\zb_2}{(|z_2|^2-1/4)^2}\lambda\left(\frac{1}{2}\left(|z_2|^2-\frac{1}{4}\right)\right)
\cr}
\end{align*}

\[\det(B)=\frac{2\zb_2(-2|z_1|^4+4|z_1|^2+1/8)}{\left(|z_1|^2-1/4\right)^4 \left(|z_2|^2-1/4\right)^2}\lambda\left(\frac{1}{2}(|z_1|^2-1/4)\right)\lambda\left(\frac{1}{2}(|z_2|^2-1/4)\right). \]
Corollary \ref{det mult} implies that the Hankel operator $H_{\det(B)}$ is compact on $A^2(\D)$.

Let $f_1(z)=(|z_1|^2-1)(|z_2|^2-1)$ and $f_2(z)=\sin\left((|z_1|^2-1)(|z_2|^2-1)\right)$. $f_1(z)$ and $f_2(z)$ are both vanishing on $L$, so they are compactness multipliers. Then Corollary \ref{allowable Levi} tells us that 

\begin{align*}
F:&=
\bordermatrix{\cr
\cr           & 
\frac{\partial f_1}{\partial z_1} & \frac{\partial f_1}{\partial z_2}
\cr & &
\cr           & 
\frac{\partial f_2}{\partial z_1} & \frac{\partial f_2}{\partial z_2}
\cr}\\
&=
\bordermatrix{\cr
\cr  &  (|z_2|^2-1) \zb_1  & (|z_1|^2-1) \zb_2
\cr  &    &
\cr  &  (|z_2|^2-1) \zb_1 \cos\left((|z_1|^2-1)(|z_2|^2-1)\right)  & (|z_1|^2-1) \zb_2\cos\left((|z_1|^2-1)(|z_2|^2-1)\right) 
\cr},
\end{align*}
 
\[\det(F)= (\zb_1-\zb_2)(|z_1|^2-1)(|z_2|^2-1)\cos\left((|z_1|^2-1)(|z_2|^2-1)\right),\]
and the Hankel operator $H_{\det(F)}$ is compact on $A^2(\D)$. 

On the other hand, to show that the Hankel operators $H_{\det(F)}$ is compact on $A^2(\D)$ one may use the result from \cite{CuckovicSahutoglu09}, which requires showing that for all analytic discs $f:\disc\rightarrow L\subset b\D$ the compositions $\det(F)\circ f(z)$ are holomorphic on $\mathbb{D}$.

\end{example}

\begin{remark}\label{vf remark}
One can find a smooth allowable vector field of type $(1,0)$ such that its integral solution is not a compactness multiplier. For example, in $\C^{2}$, let $\Omega$ be a bounded smooth pseudoconvex domain where the $\overline{\partial}$-Neumann operator not compact, that is, $\D$ be a domain with boundary (Levi flat) foliated with analytic discs. 
Consider, in a \emph{special boundary chart} see \cite[pg.13]{StraubeBook}, the $Y=0\cdot \frac{\partial}{\partial z_1}+1\cdot \frac{\partial}{\partial z_2}$ which is allowable vector field of type $(1,0)$. $Y$ is allowable vector field because $u\in \cC_{(0,1)}^{\infty}(\D)\cap\text{dom}(\dbar^*)$ and so the normal component of $u$, that is $u_2$, vanishes on $L^2$. 
On the other hand, its integral solution $y=c+z_2$, where $c$ is a constant, can not be a compactness multiplier, in general. In Example \ref{ex1}, if $y=c+z_2$ is a compactness multiplier then because $y$ does not vanish on the set $L$, there would exists compactness estimate on $B(p,r)\cap\D$, where $p\in L$ and $B(p,r)$ is a ball. This would contradict \cite[Proposition 9]{FuStraube2001} (see also \cite{SahutogluStraube2006}).

\end{remark}

\begin{example}\label{hankel example}
If a Hankel operator $H_{f}$ is compact on $A^2(\D)$ then $\partial f$ is not necessarily an allowable vector field.
$H_{z_1}\equiv 0$, so is compact on $A^2(\D)$ for any domain $\D\subset\C^n$, but $z_1$ is not a compactness multiplier on the domain defined in Example \ref{ex1}, since $z_1\not=0$ on $L$. Moreover, $\partial z_1=1\cdot\frac{\partial}{\partial z_1}+0\cdot\frac{\partial}{\partial z_2}$ is not an allowable vector field on the domain $\D$ in the same example. 
\end{example}

\begin{remark}\label{matrix remark}
One can find a matrix whose determinant is a compactness multiplier, but the matrix is not an allowable one. From the elementary algebra perspective this is because, the determinant map is a group homomorphism map from the algebra of square matrices (under matrix multiplication) to the algebra of real numbers (under multiplication), but it is not a group isomorphism.
In $\C^{2}$ consider $\Omega$ bounded smooth pseudoconvex domain where the $\overline{\partial}$-Neumann operator is not compact, let $\D$ be a domain with boundary (Levi flat) foliated with analytic discs. Then, in a \emph{special boundary chart}, on the boundary, we will have $u_{2}$, the normal component of the form $u=u_1d\zb_1+u_2d\zb_2$ as $0$. Now consider the matrix 
\[A=\bordermatrix{& & \cr  &1  &0  \cr   &0 &r \cr}.\] 
Then $\det(A)=r$. We know that the defining function $r$ is a compactness multiplier, so does $\det(A)$. 

Assume that $A$ is allowable, that is, each row of the matrix $A$ is allowable row. However, if the first row of the matrix $A$, $(A_{11},A_{12})$ is allowable then we have $\forall \varepsilon>0$ $\exists C_{\varepsilon}>0$ such that  
\[\norm{u_{1}}^2=\norm{u}^2 \leq \varepsilon (\norm{\overline{\partial}u}^{2}+\norm{\overline{\partial}^{*}u}^{2})+ C_{\varepsilon}\norm{u}_{-1}^{2}
\]
$\forall u\in \cC_{(0,1)}^{\infty}(\D)\cap \text{dom}(\dbar^*)$.  
This implies the existence of compactness on the domain $\D$ which contradicts \cite[Proposition 9]{FuStraube2001} (see also \cite{SahutogluStraube2006}). 
Thus, the matrix $A$ is not allowable although its determinant is a compactness multiplier.

\end{remark}

\section{Acknowledgements}
The authors would like to thank the anonymous referee for the useful comments. The authors also would like to thank S{\"o}nmez {\c{S}}ahuto{\u{g}}lu for valuable comments on an earlier version of this manuscript. 



\begin{thebibliography}{Koh79}

\bibitem[Axl86]{Axler1986}
Sheldon Axler, \emph{The {B}ergman space, the {B}loch space, and commutators of
  multiplication operators}, Duke Math. J. \textbf{53} (1986), no.~2, 315--332.

\bibitem[{\c{C}}el08]{CelikPhD08}
Mehmet {\c{C}}elik, \emph{Contributions to the compactness theory of the
  (part)-{N}eumann operator}, ProQuest LLC, Ann Arbor, MI, 2008, Thesis
  (Ph.D.)--Texas A\&M University. \MR{2711959}

\bibitem[CS01]{ChenShawBook}
So-Chin Chen and Mei-Chi Shaw, \emph{Partial differential equations in several
  complex variables}, AMS/IP Studies in Advanced Mathematics, vol.~19, American
  Mathematical Society, Providence, RI, 2001.

\bibitem[{\c{C}}S09a]{CelikStraube09}
Mehmet {\c{C}}elik and Emil~J. Straube, \emph{Observations regarding
  compactness in the {$\overline{\partial}$}-{N}eumann problem}, Complex Var.
  Elliptic Equ. \textbf{54} (2009), no.~3-4, 173--186.

\bibitem[{\v{C}}{\c{S}}09b]{CuckovicSahutoglu09}
{\v{Z}}eljko {\v{C}}u{\v{c}}kovi{\'c} and S{\"o}nmez {\c{S}}ahuto{\u{g}}lu,
  \emph{Compactness of {H}ankel operators and analytic discs in the boundary of
  pseudoconvex domains}, J. Funct. Anal. \textbf{256} (2009), no.~11,
  3730--3742.

\bibitem[{\c{C}}{\c{S}}12]{CelikSahutoglu2012}
Mehmet {\c{C}}elik and S{\"o}nmez {\c{S}}ahuto{\u{g}}lu, \emph{On compactness
  of the {$\overline\partial$}-{N}eumann problem and {H}ankel operators}, Proc.
  Amer. Math. Soc. \textbf{140} (2012), no.~1, 153--159.

\bibitem[{\u{C}}{\c{S}}13]{CuckovicSahutoglu13}
{\u{Z}}eljko {\u{C}}u{\u{c}}kovi{\'c} and S{\"o}nmez {\c{S}}ahuto{\u{g}}lu,
  \emph{Axler-{Z}heng type theorem on a class of domains in {$\Bbb{C}\sp n$}},
  Integral Equations Operator Theory \textbf{77} (2013), no.~3, 397--405.

\bibitem[{\c{C}}{\c{S}}14a]{CelikSahutoglu2014}
Mehmet {\c{C}}elik and S{\"o}nmez {\c{S}}ahuto{\u{g}}lu, \emph{Compactness of
  the {$\overline\partial$}-{N}eumann operator and commutators of the {B}ergman
  projection with continuous functions}, J. Math. Anal. Appl. \textbf{409}
  (2014), no.~1, 393--398.

\bibitem[{\v{C}}{\c{S}}14b]{CuckovicSahutoglu14}
{\v{Z}}eljko {\v{C}}u{\v{c}}kovi{\'c} and S{\"o}nmez {\c{S}}ahuto{\u{g}}lu,
  \emph{Compactness of products of {H}ankel operators on convex {R}einhardt
  domains in {$\Bbb C\sp 2$}}, New York J. Math. \textbf{20} (2014), 627--643.

\bibitem[D'A93]{D'AngeloBook1992}
John~P. D'Angelo, \emph{Several complex variables and the geometry of real
  hypersurfaces}, Studies in Advanced Mathematics, CRC Press, Boca Raton, FL,
  1993.

\bibitem[FS01]{FuStraube2001}
Siqi Fu and Emil~J. Straube, \emph{Compactness in the
  {$\overline\partial$}-{N}eumann problem}, Complex analysis and geometry
  ({C}olumbus, {OH}, 1999), Ohio State Univ. Math. Res. Inst. Publ., vol.~9, de
  Gruyter, Berlin, 2001, pp.~141--160.

\bibitem[Has01]{Haslinger01}
Friedrich Haslinger, \emph{The canonical solution operator to
  {$\overline\partial$} restricted to {B}ergman spaces}, Proc. Amer. Math. Soc.
  \textbf{129} (2001), no.~11, 3321--3329 (electronic).

\bibitem[HJ13]{MatrixAnalysis}
Roger~A. Horn and Charles~R. Johnson, \emph{Matrix analysis}, second ed.,
  Cambridge University Press, Cambridge, 2013.

\bibitem[H{\"o}r65]{Hormander65}
Lars H{\"o}rmander, \emph{{$L\sp{2}$} estimates and existence theorems for the
  {$\overline\partial $}\ operator}, Acta Math. \textbf{113} (1965), 89--152.

\bibitem[Koh79]{Kohn79}
J.~J. Kohn, \emph{Subellipticity of the {$\bar \partial $}-{N}eumann problem on
  pseudo-convex domains: sufficient conditions}, Acta Math. \textbf{142}
  (1979), no.~1-2, 79--122.

\bibitem[{\c{S}}ah12]{Sahutoglu12}
S{\"o}nmez {\c{S}}ahuto{\u{g}}lu, \emph{Localization of compactness of {H}ankel
  operators on pseudoconvex domains}, Illinois J. Math. \textbf{56} (2012),
  no.~3, 795--804.

\bibitem[{\c{S}}S06]{SahutogluStraube2006}
S{\"o}nmez {\c{S}}ahuto{\u{g}}lu and Emil~J. Straube, \emph{Analytic discs,
  plurisubharmonic hulls, and non-compactness of the
  {$\overline\partial$}-{N}eumann operator}, Math. Ann. \textbf{334} (2006),
  no.~4, 809--820.

\bibitem[Str08]{Straube2008}
Emil~J. Straube, \emph{A sufficient condition for global regularity of the
  {$\overline\partial$}-{N}eumann operator}, Adv. Math. \textbf{217} (2008),
  no.~3, 1072--1095.

\bibitem[Str10]{StraubeBook}
\bysame, \emph{Lectures on the $\mathcal{L}^{2}$-{S}obolev theory of the
  $\overline\partial$-{N}eumann problem}, ESI Lectures in Mathematics and
  Physics, vol.~7, European Mathematical Society (EMS), Z\"urich, 2010.

\bibitem[{\c{S}}Z]{ZeytuncuSahutoglu16}
S{\"o}nmez {\c{S}}ahuto{\u{g}}lu and Yunus~E. Zeytuncu, \emph{On compactness of
  {H}ankel and the {$\overline\partial$}-{N}eumann operators on {H}artogs
  domains in $\mathbb{C}^2$}, to appear in J. Geom. Anal.

\end{thebibliography}
\end{document}